\documentclass[12pt,a4paper,reqno]{amsart}

\usepackage[latin1]{inputenc}
\usepackage{amssymb,amsmath,amstext,amsthm,amsfonts, mathtools}
\usepackage{graphicx}
\usepackage{multicol}
\usepackage{subfigure}
\usepackage{tikz}
\usepackage{hyperref}

\usepackage{a4wide}
\newcommand{\R}{\mathbb{R}}

\newcommand*{\rc}{\mathcal{C}}

\def\inte{\operatorname{int}}

\newtheorem{maintheorem}{Theorem}

\newtheorem{lemma}{Lemma}[section]

\newtheorem{corollary}[lemma]{Corollary}
\newtheorem{proposition}[lemma]{Proposition}

\theoremstyle{remark}
\newtheorem{remark}[lemma]{Remark}

\thanks{JFA was partially funded by Funda\c c\~ao Calouste Gulbenkian, by the
European Regional Development Fund through the program COMPETE and by the Portuguese Government through FCT  under the projects PEst-C/MAT/UI0144/2013 and PTDC/MAT/120346/2010. AP and EV were partially supported by MEC grant  MTM2011-22956, and the Foundation for the promotion in Asturias of the
Scientific  and Technologic Research BP12-123}

\begin{document}
\title[Statistical stability for piecewise expanding  maps]{Statistical stability for multidimensional\\ piecewise expanding  maps}
\date{\today}

\author[J. F. Alves]{Jos\'{e} F. Alves}
\address{Jos\'{e} F. Alves\\ Centro de Matem\'{a}tica da Universidade do Porto\\ Rua do Campo Alegre 687\\ 4169-007 Porto\\ Portugal}
\email{jfalves@fc.up.pt} \urladdr{http://www.fc.up.pt/cmup/jfalves}

\author[A. Pumari\~no]{Antonio Pumari\~no}
\address{Antonio Pumari\~no\\ Departamento de Matem\'aticas, Facultad de Ciencias de la Universidad de Oviedo, Calvo Sotelo s/n, 33007 Oviedo, Spain.}
\email{apv@uniovi.es} 

\author[E. Vigil]{Enrique Vigil}
\address{Enrique Vigil\\ Departamento de Matem\'aticas, Facultad de Ciencias de la Universidad de Oviedo, Calvo Sotelo s/n, 33007 Oviedo, Spain.}
\email{vigilkike@gmail.com}
\maketitle




\begin{abstract}
We present sufficient conditions for  the (strong) statistical stability of some  classes of multidimensional piecewise expanding maps. As a consequence we get that a certain natural two-dimensional extension of the classical one-dimensional family of tent maps is statistically stable.
\end{abstract}


\tableofcontents

\section{Introduction}
Along this paper we deal with  multidimensional piecewise expanding maps  defined in some compact subset of an Euclidean space. A first approach to this topic has been made in  dimension one by Lasota and Yorke in \cite{LY73}, where they proved the existence of absolutely continuous invariant probability measures for a class of piecewise $C^2$ expanding maps of the interval. The extension to higher dimensions in general is a very delicate question, mostly because of the intricate geometry of the domains of smoothness and their images under iterations.   In the last decades many results have appeared in the literature with several different approaches.  A first  result  in dimension two was obtained by Keller in~\cite{K79}. In the multidimensional case, Gora and Boyarski in \cite{GB89} proved the existence of absolutely continuous invariant probability measures for maps with a finite number of domains of smoothness under some condition of no cusps in the domains of smoothness. This result was later extended by Adl-Zarabi in \cite{A96} to  piecewise expanding maps allowing cusps in the domains of smoothness, and  by Alves in~\cite{A00} to piecewise expanding maps with countably many domains of smoothness. Similar results were drawn in the particular case of  piecewise linear maps by Buzzi and  Tsujii in \cite{B99a,T01b}, and by the same authors in the case of  piecewise real analytic expanding maps of the plane in \cite{B00a,T00}. 
A general result on the existence of absolutely continuous invariant probability measures in any finite dimension was given by Saussol in \cite{S00}  for $C^{1+}$ piecewise expanding maps with infinitely many domains of smoothness under some control on the accumulation  of discontinuities  under iterations of the map.

Many concrete situations lead to the appearance of families of piecewise expanding maps under conditions that guarantee the existence of absolutely continuous invariant probability measures, in many cases this probability measure being unique. It is then a natural question trying to decide weather these measures depend continuously on the dynamics, i.e the statistical stability of those maps. This question was addressed  in \cite{AV02} for certain robust classes of maps with non-uniform expansion; see also \cite{A04}.
In \cite{PRT14}, the authors consider a one-parameter family $(\Lambda_t)_t$ of two-dimensional piecewise linear maps defined on a triangle in $\R^2$. This family $(\Lambda_t)_t$ is closely related to the family of limit return maps arising when certain three-dimensional homoclinic bifurcations take place; see \cite{T01}.
It is showed in \cite[Proposition 6.1]{PRT14} and
the comments following  it that $\Lambda_t$ becomes the best choice in the piecewise linear setting for describing the dynamics of the original limit return maps
given in \cite{T01} when the unstable manifold of the saddle point has dimension two.

The results in  the present  paper have a twofold aim:
to give sufficient conditions for the statistical stability of some general classes of multidimensional piecewise expanding maps; and to prove the statistical stability of the family of maps  introduced in \cite{PRT14}. This last result will be obtained as an application of our general result.

\subsection{Statistical stability}
In this work we consider discrete-time dynamical systems defined  in a compact region $R\subset \R^d$, for some $d\ge1$. Given  a measurable map $\phi:R\to R$, we say that a  probability measure $\mu$ on the Borel sets of $R$ is  \emph{$\phi$-invariant} if $\mu(\phi^{-1}(A))=\mu(A)$, for any Borel set $A\subset R$. If $\mu(A)=0$ whenever $m(A)=0$, where $m$ denotes the Lebesgue measure on the Borel sets of $\R^d$, then $\mu$ is called \emph{absolutely continuous}. In this case, there exists an $m$-integrable  function $h\ge 0$, usually denoted $d\mu/dm$ and  called the density of $\mu$ with respect to $m$,  such that for any Borel set $A\subset R$ we have
 $\mu(A)=\int_Ahdm$. A $\phi$-invariant probability measure $\mu$ is called ergodic if $\mu(A)\mu(R\setminus A)=0$, whenever $\phi^{-1}(A)=A$. As a consequence of Birkhoff's Ergodic Theorem we have that any $\phi$-invariant absolutely continuous ergodic probability measure $\mu$ is a \emph{physical measure}, meaning that for a subset of points $x\in R$ with positive Lebesgue measure we have
 $$\lim_{n\to\infty}\sum_{j=0}^{n-1}f(\phi^j(x))=\int fd\mu$$
 for all continuous $f:R\to \R$.

Let $I$ be a metric space and $(\phi_t)_{t\in I}$ a family of maps $\phi_t:R\rightarrow R$  such that each $\phi_t$ has some absolutely continuous $\phi_t$-invariant probability measure. We say that  the family $(\phi_{t})_{t\in I}$ is \emph{statistically stable} if for any $t_0\in I$, any choice of a sequence $(t_n)_n$ in $I$ converging to $t_0$ and any  choice of  a sequence of absolutely continuous $\phi_{t_n}$-invariant probability measures $(\mu_{t_n})_n$, then any accumulation point of   the sequence of  densities $d\mu_{t_n}/dm$ must  converge in the $L^1$-norm to the density of an absolutely continuous $\phi_{t_0}$-invariant probability measure. Of course, when each $\phi_t$ has a unique absolutely continuous invariant probability measure~$\mu_t$, then  statistical stability means that  $d\mu_t/dm$ converges in the $L^1$-norm to $d\mu_{t_0}/dm$ when~$t\to t_0$. A strictly weaker notion of statistical stability may be given if we assume only weak* convergence of the measures $\mu_t$ to $\mu_{t_0}$ when $t\to t_0$.

\subsection{Piecewise expanding maps}

Here we state precisely sufficient
conditions for the statistical stability of  certain higher dimensional families of $C^2$
piecewise expanding maps with countably many domains of
smoothness. We follow the approach  in~\cite{A00} which, on its turn, was inspired in \cite{GB89}. 

Let $R$ be a compact set in $\R^d$, for some $d\ge 1$. For each $1\le p\le \infty$ we denote by  $L^p(R)$ the Banach space of functions in $L^p(m)$ with support contained in $R$, endowed with the usual norm $\|\quad\|_p$.  
Let $\phi: R\to R$ a map for which there is a (Lebesgue mod 0) partition
$\{R_i\}_{i=1}^{\infty}$ of $R$ such that each $R_i$ is a
closed domain with piecewise $C^2$ boundary of finite
$(d-1)$-dimensional measure and 
$\phi_i=\phi|R_i$ is a $C^2$ bijection from $\inte(R_i)$, the interior of~$R_i$, onto its image with a $C^2$ extension to $R_i$. We say that $\phi$ is  {\it piecewise
expanding} if
\begin{itemize}
\item[(P$_1$)] there is $0<\sigma<1$ such that for every $i\geq 1$ and $x\in\inte(\phi(R_i))$
$$\| D\phi_i^{-1}(x)\|
<\sigma.$$ 
 \end{itemize}
We say that  $\phi$ has {\it
bounded distortion} if
\begin{itemize} \item[(P$_2$)] there
is $D\ge 0$ such that for every $i\geq 1$ and $x\in\inte(\phi(R_i))$ 
$$\frac{\left\|
D\left(J\circ\phi^{-1}_i\right)(x)\right\|}{\left|J\circ\phi^{-1}_i\right(x)|}\le D,$$
where $J$ is the Jacobian of $\phi$.
 \end{itemize}
Finally, we say that $\phi$ has \emph{long branches} 
 if 
\begin{enumerate}
 \item[(P$_3$)]   there are $\beta,\rho>0$ and for each $i\ge 1$ there is a $C^1$
unitary vector field $X_i$ in $\partial \phi(R_i)$ such that:
\begin{enumerate}
\item[(a)]   the segments joining each $x\in\partial \phi(R_i)$ to
$x+\rho X_i(x)$ are pairwise disjoint and contained in $\phi(R_i)$, and their union forms a neighborhood of $\partial \phi(R_i)$ in $\phi(R_i)$.
\item[(b)] for every $x\in\partial \phi(R_i)$
and $v\in T_x\partial \phi(R_i)\setminus\{0\}$ the angle $\angle(v,X_i(x))$ between $v$ and $X_i(x)$ 
satisfies $|\sin\angle(v,X_i(x))|\geq
\beta$. 
\end{enumerate}
 \end{enumerate}
Here we assume that at the singular points $x\in\partial \phi(R_i)$
where $\partial \phi(R_i)$ is not smooth the vector $X_i(x)$ is a
common $C^1$ extension of $X_i$ restricted to each
$(d-1)$-dimensional smooth component of $\partial \phi(R_i)$ having
$x$ in its boundary. We also assume that the tangent space
of any such singular point $x$ is the union of the tangent
spaces to the  $(d-1)$-dimensional smooth components it
belongs to. 

In the one-dimensional case $d=1$, condition (P$_3$)(a)  is clearly satisfied once we take the sets in the partition of $R$ as being intervals whose images $\phi(R_i)$ have sizes uniformly bounded away from zero. Additionally, condition~(P$_3$)(a) 
always  holds in dimension one, since $ \phi(R_i)$ is a 0-dimensional manifold and so $T_x\partial \phi(R_i)=\{0\}$ for any $x\in \phi(R_i)$. In this case we can even take the optimal value $\beta=1$; see Remark~\ref{re.beta}.

\begin{maintheorem}\label{expanding} Let $I$ be a metric space and $(\phi_t)_{t\in I}$ a family of $C^2$
 piecewise expanding maps $\phi_t:R\to R$ with bounded distortion and
long branches. Assume that for each $t\in I$
\begin{enumerate}
\item for each continuous $f:R\to\R$ we have $\|f\circ \phi_{t'}-f\circ\phi_t\|_d\to0$ when $t'\to t$;
\item there exist $0<\lambda<1$ and $K>0$ for which $$\sigma_t\left(1+\frac1{\beta_t}\right)\le \lambda\quad\text{and}\quad D_t+\frac{1}{\beta_t\rho_t}+\frac{D_t}{\beta_t}  \le K,$$ 
where $\sigma_t, D_t,\beta_t,\rho_t$ are constants for which (P$_1$), (P$_2$) and (P$_3$) hold for $\phi_t$.
\end{enumerate}
 
Then
$(\phi_t)_{t\in I}$ is  statistically stable.

 \end{maintheorem}
It follows from  \cite[Section 5]{A00} that under the assumptions above each $\phi_t$ has
a finite number of ergodic absolutely continuous invariant probability measures.
The proof of this result uses the space of functions of bounded variation in $\R^d$, which are known to belong to the space $L^p(R)$, with $p=d/(d-1)$; see~\eqref{bvp} below. Observing that $1/p+1/d=1$, this makes the choice of the norm $\|\quad\|_d$ in condition (1) less mysterious; see the proof of Lemma~\ref{p.igual}.
Notice that condition (1) in Theorem~\ref{expanding}  holds whenever the maps $\phi_t$ are continuous and $\phi_t$ depends continuously  (in the $C^0$-norm) on $t\in I$.

\subsection{Two-dimensional tent maps} 

Here we present  the family of maps introduced in
  \cite{PRT14} and give some result on its statistical stability. We define the family of maps $\Lambda_{t}:T\to T$ on the triangle ${T}={T}_0\cup {T}_1$, where
\begin{equation}\label{algo}
{T}_0=\{(x,y):0 \leq x \leq 1, \ 0 \leq y \leq x \},\quad
{T}_1=\{(x,y):1 \leq x \leq 2, \ 0 \leq y \leq 2-x \},
\end{equation}
and
\begin{equation}\label{familyt}
\Lambda_t(x,y)= \left\{
\begin{array}{ll}
(t(x+y),t(x-y)), & \mbox{if }    (x,y)\in {T}_0;\\
(t(2-x+y),t(2-x-y)), & \mbox{if } (x,y)\in {T}_1.%
\end{array}
\right.
\end{equation}
The domains $T_0$ and $T_1$ are separated by a straight line segment $\mathcal C=\{(x_1,x_2)\in T: x_1=1\}$ that we call the \emph{critical set} of $\Lambda_t$.

As shown in \cite{PT06}, the map $\Lambda_1$ displays the same properties of the one-dimensional tent map $\lambda_2(x)=1-2|x|$. Among them, the consecutive
pre-images $\{\Lambda_1^{-n}(\rc)\}_{n\in \mathbb{N}}$ of the critical line $\rc$ define a sequence of partitions (whose diameter tends to zero as $n$ goes to infinity)
of $T$ leading them to conjugate $\Lambda_1$ to a one sided shift with two symbols. Hence, it easily follows that
$\Lambda_1$ is transitive in $T$. Furthermore, for every point $(x_0,y_0)\in T$ whose orbit never hits the
critical line the Lyapounov exponent of $\Lambda_1$ along the orbit of $(x_0,y_0)$ is positive (and coincides with $\frac{1}{2}\log{2}$) in all nonzero direction.
Finally, it can be constructed an absolutely continuous ergodic invariant probability measure for $\Lambda_1$; see  \cite{PT06}. Because of this, $\Lambda_1$ was called  the \textit{two-dimensional tent map}. Since the parameter $t$ in (\ref{familyt}) essentially gives the rate of expansion for $\Lambda_t$ (playing the same roll of the parameter $a$ for $\lambda_a(x)=1-a|x|$), the family $(\Lambda_t)_t$ can be considered as a natural extension of the one-dimensional family of tent maps and naturally called a  \textit{family of two-dimensional tent maps}.

The results obtained in \cite{PT06} for  $t=1$ were extended to a larger set of parameters. More precisely,  it was proved  in \cite{PRT14a}  that
for each $t \in [\tau,1] $, with $ \tau=\frac{1}{\sqrt{2}}(\sqrt{2}+1)^\frac{1}{4}\approx 0.882,$ the map $\Lambda_t$  exhibits a \emph{strange attractor} $A_t\subset T$:  $\Lambda_t$ is (strongly) transitive in $A_t$, the periodic orbits are dense in $A_t$, and  there exists a dense orbit in $A_t$ with two positive Lyapunov exponents. Furthermore,  $A_t$ supports a unique absolutely continuous $\Lambda_t$-invariant  ergodic probability measure~$\mu_t$.  
As an application  of Theorem~\ref{expanding} we shall obtain the following result.


\begin{maintheorem}\label{main} The family $(\Lambda_t)_{t\in [\tau,1]}$ is statistically stable.
\end{maintheorem}

As each $\Lambda_t$ has a unique absolutely continuous invariant probability measure $\mu_t$, the statistical stability means in this case  that  $d\mu_t/dm$ converges in the $L^1$-norm to $d\mu_{t_0}/dm$ when~$t\to t_0$, for each $t_0\in[\tau,1]$.

\section{Functions of bounded variation}\label{variation}

The main ingredient for the proof of Theorem~\ref{expanding} is the
notion of variation for functions in multidimensional spaces. We
adopt the definition given in \cite{G84}. Given $f\in
L^1(\R^d)$ with compact support
 we define the {\it variation} of $f$ as
$$V(f)=\sup\left\{\int_{\R^n}f\mbox{div}(g)dm\,:\,g\in
C_0^1(\R^d,\R^d)\text{ and }  \|g\|\leq 1\right\},$$ where
$C_0^1(\R^d,\R^d)$ is the set of $C^1$ functions from
$\R^d$ to $\R^d$ with compact support, $\mbox{div}(g)$ is the divergence of $g$  and $\|\quad\|$ is
the sup norm in $C_0^1(\R^d,\R^d)$. 
%
%
Given a bounded set $R\subset\R^d$ we consider the space of {\em bounded variation} functions in
$L^1(R)$
$$BV(R)=\left\{ f\in L^1(R):V(f)<+\infty\right\}.$$ 
Contrarily to the classical one-dimensional definition of bounded variation,  a multidimensional bounded variation function  need not to be
bounded; see  \cite{GB92}. However, by  Sobolev's Inequality
(see e.g. \cite[Theorem 1.28]{G84}) there is some constant $C>0$ (only
depending on the dimension $d$) such that for any $f\in
BV(R)$
 \begin{equation}\label{bvp}
 \left(\int|f|^pdm_d\right)^{1/p}\leq C\;V(f),
\quad \mbox{with}\quad p=\frac{d}{d-1}.
\end{equation}
This in
particular gives $BV(R)\subset L^p(R)$. 
We shall use the following properties of bounded variation functions whose proofs 
 may be
found in  \cite{EG92} or \cite{G84}:
 \begin{itemize}
\item[(B$_1$)] $BV(R)$ is dense in $L^1(R)$;
 \item[(B$_2$)]
if $(f_k)_{k}$ is a sequence in $BV(R)$
converging to $f$ in the $L^1$-norm, then $V(f)\leq
\liminf_k V(f_k)$;
\item[(B$_3$)] if $(f_k)_k$ is a sequence
in $BV(R)$ such that $\big(\|f_k\|_1\big)_k$ and
$\big(V(f_k)\big)_k$ are bounded, then $(f_k)_k$
has some subsequence converging in the $L^1$-norm to a
function in $BV(R)$. \end{itemize}

\section{Piecewise expanding
maps}\label{multidimensional}
In this section we prove Theorem~\ref{expanding}.
Let $\{R_i^t\}_{i=1}^\infty$ be  the  domains of smoothness of   $\phi_t$ with $t\in I$ satisfying  the assumptions of Theorem~\ref{expanding} and define $\phi_{t,i}=\phi_t\vert_{R_i^t }$ for all $i\ge 1$. For each $t\in I$ we consider  the  \emph{Perron-Frobenius operator}
$$P_t:L^1(R)\longrightarrow L^1(R)$$ defined for $f\in L^1(R)$ as 
$$P_t
f=\sum_{i=1}^{\infty}\frac{f\circ\phi_{t,i}^{-1}}{|J\circ\phi_{t,i}^{-1}|}
\chi_{\phi_t(R_i^t)}.$$
 It is well known that the
following two properties hold for each $P_t$.

\begin{itemize} \item[(C$_1$)] $\|P_t f\|_1\leq \| f\|_1$ for
every $f\in L^1(R)$; \item[(C$_2$)] $P_t f=f$ if and only if
$f$ is the density of an absolutely continuous $\phi_t$-invariant
probability measure. \end{itemize}

Considering  $0<\lambda<1$ and $K>0$ as in the statement of Theorem~\ref{expanding}, the proof of the next lemma follows immediately from~\cite[Lemma~5.4 \& Lemma~5.5]{A00} with 
$$K_1=K\sum_{j=0}^\infty
\lambda^j.$$

\begin{lemma} \label{lema3} Given $t\in I$ and $j\geq 1$ we have for each $f\in BV(R)$ 
$$V(P_t^jf)\leq \lambda^jV(f)+K_1\|f\|_1.$$ 
\end{lemma}

\begin{remark}\label{re.beta}
The proof of  \cite[Lemma~5.4]{A00} uses \cite[Lemma 3]{GB89} applied to the sets $S=\phi(R_i)$, which gives for a function $f\in C^1(S)$,  
\begin{equation}\label{eq.varia}
\int_{\partial
S}|f|dm\leq 
\frac{1}{\beta}\left(\frac{1}{\rho}\int_{S}|f|dm+
\int_{S}\|Df\|dm\right).
\end{equation}
In the one-dimensional case we have for any interval $S$ and $x\in S$
$$f(x)\leq 
\frac{1}{|S|}\int_{S}|f|dm+
\int_{S}|Df|dm,$$
which yields  a formula similar to \eqref{eq.varia} in the one-dimensional case with $\beta =1$.
\end{remark}

In the proof of the result below we
follow some standard arguments with functions of bounded variation, namely those used    in
\cite{LY73} for the one-dimensional case. 

\begin{proposition} \label{pr.dakhc}Given $t\in I$ and
$f\in L^1(R)$ the sequence $1/n\sum_{j=0}^{n-1}P^j_t f$ has some accumulation point in the $L^1$-norm. Moreover, any such accumulation  point   belongs to $BV(R)$ and has variation bounded by $ 4K_1\|f\|_1$.
\end{proposition}

\begin{proof}
Given $f\in L^1(R)$, by property (B$_1$) we may consider a sequence of functions $(f_k)_k$ in $BV(R)$
converging to $f$ in the $L^1$-norm. With no loss of generality we may assume that $\|f_k\|_1\leq 2\|f\|_1$ for every $k\geq 1$.  It follows from Lemma~\ref{lema3} that for each $k\geq 1$ and large $j$ we have $$ V(P^j_tf_k)\leq
\lambda^jV(f_k)+K_1\|f_k\|_1\leq 3K_1\|f\|_1.$$ 
So, for large $n$  we have
$$V\left(\frac{1}{n}\sum_{j=0}^{n-1}P^j_t f_k\right)\leq
4K_1\|f\|_1.$$
Using that $\|f_k\|_1\leq 2\|f\|_1$ for every $k\geq 1$, it easily follows from (C$_1$) that
$$\left\|\frac{1}{n}\sum_{j=0}^{n-1}P^j_t f_k\right\|_1\leq
2\|f\|_1.$$
Then it follows from (B$_3$) that there exists some
$g_k\in BV(R)$ and a sequence $(n_i)_i$ such that
$1/n_i\sum_{j=0}^{n_i-1}P_t^j f_k$ converges in the
$L^1$-norm to $g_k$ as $i$ goes to $+\infty$.
Moreover, by (B$_2$) we have $V(g_k)\leq 4K_1\|f\|_1$ for every $k\ge 1$.
Hence, we may apply the same argument to the sequence
$(g_k)_k$ and obtain a subsequence $(k_i)_i$
such that $(g_{k_i})_i$ converges in the $L^1$-norm to
some $g\in BV(R)$ with $V(g)\leq 4K_1\|f\|_1$. 
Hence, there must be some
sequence $(n_\ell)_\ell$ converging to $+\infty$  for which
$1/n_\ell\sum_{j=0}^{n_\ell-1}P^j_t f_{k_\ell}$ converges to $g$
in the $L^1$-norm as $\ell\to +\infty$. On the other
hand, 
$$\left\|\frac{1}{n_\ell}\sum_{j=0}^{n_\ell-1}\big(P^j_t
f_{k_\ell}-P^j_tf\big)\right\|_1\leq\frac{1}{n_\ell}\sum_{j=0}^{n_\ell-1}
\left\|f_{k_\ell}-f\right\|_1=\left\|f_{k_\ell}-f\right\|_1$$ and
this last term goes to 0 as $\ell\rightarrow +\infty$. This
clearly gives that ${1}/{n_\ell}\sum_{j=0}^{n_\ell-1}P^j_t f$
converges to $g$ in the $L^1$-norm.

To prove the second part of the lemma, consider some subsequence of $1/n\sum_{j=0}^{n-1}P^j_t f$ converging to $f_0$ in the $L^1$-norm. Taking that subsequence playing the role of the whole sequence in the argument above we easily see that $f_0$ satisfies the conclusion by uniqueness of the limit.
\end{proof}

\begin{corollary}\label{co.bueno}
If $h_t$ is the density of an absolutely continuous $\phi_t$-invariant probability measure, then $h_t\in BV(R)$ and $V(h_t)\le 4K_1$.
\end{corollary}

\begin{proof}
Take $h_t$ the density of an absolutely continuous $\phi_t$-invariant probability measure. We have from property (C$_2$) that $P_t^jh_t=h_t$ for all $j\ge 1$. This implies that the sequence $1/n\sum_{j=0}^{n-1}P^j_t h_t$ is constant and equal to $h_t$, and so the result follows. 
\end{proof}

Now we are in conditions to conclude the proof of Theorem~\ref{expanding}. 

Let $(t_n)_n$ be a sequence in $I$ converging to some $t_0\in I$. Assume that for each $n\ge 1$ we have an absolutely continuous $\phi_{t_n}$-invariant probability measure $\mu_n$ and consider $$h_n=\frac{d\mu_n}{dm}.$$
Using the fact that each $h_n$ is the density of a probability measure and Corollary~\ref{co.bueno} we  have for all $n\ge 1$
$$\|h_n\|_1= 1\quad\text{and}\quad V(h_n)\le 4K_1.$$ 
Hence, by (B$_2$) and (B$_3$) there exists $h_0\in BV$ with $V(h_0)\le 4K_1$ such that the sequence $(h_n)_n$ converges to $h_0$ in the $L^1$-norm. 
Let $\mu_0$ be the probability measure in $R$ whose density with respect to $m$ is $h_0$.
Theorem~\ref{expanding} is now a consequence of the following lemma.

\begin{lemma}\label{p.igual}
 $\mu_0$ is a $\phi_{t_0}$-invariant measure.
\end{lemma}
 \begin{proof}
 Since the sequence $(h_n)_n$ converges to $h_0$ in the $L^1$-norm, it easily follows   that
$(\mu_n)_n$ converges to $\mu_0$ in the weak*
topology. Thus, given any $f\colon R\rightarrow\R$
continuous we have
 $$
\int fd\mu_n\longrightarrow \int
fd\mu,\quad\mbox{when } n\rightarrow\infty.
 $$ On
the other hand, since $\mu_n$ is $\phi_{t_n}$-invariant
we have
 $$
 \int fd\mu_n=\int
(f\circ\phi_{t_n})d\mu_n,\quad\mbox{for every }n.
 $$
It is enough to prove that
 \begin{equation*}
 \int (f\circ\phi_{t_n})d\mu_n
 \longrightarrow \int (f\circ\phi_{t_0}) d\mu,
 \quad\mbox{when } n\rightarrow\infty.
 \end{equation*}
 We have
 \begin{eqnarray*}
 \lefteqn{
\left|\int (f\circ\phi_{t_n})d\mu_n
 -\int (f\circ\phi_{t_0})
    d\mu\right|  }\\
& &\hspace{1cm}
\le  \left|\int (f\circ\phi_{t_n})d\mu_n
 - \int (f\circ\phi_{t_0}) d\mu_n\right|
+\left|\int (f\circ\phi_{t_0}) d\mu_n
 - \int (f\circ\phi_{t_0}) d\mu\right|\\
 & &\hspace{1cm}
\le  \int \left|f\circ\phi_{t_n}-f\circ\phi_{t_0}\right| d\mu_n
+\left|\int (f\circ\phi_{t_0}) d\mu_n
 - \int (f\circ\phi_{t_0}) d\mu\right|\\
  & &\hspace{1cm}
= \int \left|f\circ\phi_{t_n}-f\circ\phi_{t_0}\right| h_n\,dm
+\left|\int (f\circ\phi_{t_0}) (h_n-h_0) \,dm\right|.
 \end{eqnarray*}
 Now using \eqref{bvp} we easily get that each $h_n\in L^p(R)$ with $p=d/(d-1)$ and
 $$\|h_n\|_p\le CV(h_n)\le 4CK_1.$$
 Observing  that $1/p+1/d=1$, then by H\"older's Inequality we get
  $$\int \left|f\circ\phi_{t_n}-f\circ\phi_{t_0}\right| h_n\,dm\le \|f\circ \phi_{t_n}-f\circ\phi_{t_0}\|_d\cdot\|h_n\|_p\le 4C_dK_1 \|f\circ \phi_{t_n}-f\circ\phi_{t_0}\|_d,$$
  and this clearly converges to zero, when $n\to+\infty$, by assumption (1) in the statement of Theorem~\ref{expanding}. On the other hand, as $f$ is bounded we have
   $$\left|\int (f\circ\phi_{t_0}) (h_n-h_0) \,dm\right|\le \|f\circ\phi_{t_0}\|_\infty\cdot\|h_n-h_0\|_1$$
 and this clearly converges to 0 when $n\to+\infty$ as well. 
 \end{proof}

 \section{Two-dimensional tent maps} 
 
 In this section we shall prove Theorem~\ref{main}. The idea is to obtain it as  a corollary of Theorem~\ref{expanding}. As observed before, each $\Lambda_t$ is \emph{strongly transitive}: any open set becomes the whole space under a finite number of iterations by $\Lambda_t$.  This implies that the  absolutely continuous $\Lambda_t$-invariant ergodic probability measure~$\mu_t$ must be unique. Moreover, any power of $\Lambda_t$ has a unique absolutely continuous invariant ergodic probability measure as well, which must  necessarily coincide with~$\mu_t$. Thus, it is enough to obtain the statistical stability for some power of the maps in our family.  
 
 We are going to see that the family $(\Lambda_t^3)_{t\in[\tau,1]}$ is in the conditions of Theorem~\ref{expanding}. Namely, each $\Lambda_t^3:T\to T$ is a $C^2$ piecewise expanding map with bounded distortion and long branches with constants $\sigma_t, D_t,\beta_t,\rho_t$  satisfying (P$_1$), (P$_2$) and (P$_3$), and
  \begin{equation}\label{eq.dt}
\sigma_t\left(1+\frac1{\beta_t}\right)\le \lambda\quad\text{and}\quad D_t+\frac{1}{\beta_t\rho_t}+\frac{D_t}{\beta_t}  \le K,
\end{equation}
  for some choice of uniform constants $0<\lambda<1$ and $K>0$; observe that as the maps $\Lambda_t^3$ are continuous then the first condition  in Theorem~\ref{expanding} is trivially satisfied.
  
  From the definition of $\Lambda_t$  in \eqref{algo} and \eqref{familyt} we obviously have that $T_0$ and $T_1$ are the domains of smoothness of $\Lambda_t$.  The map $\Lambda_t$ is piecewise linear with 
   \[
   D\Lambda_t(x)=
\left(
\begin{array}{cc}
   t  & t  \\
   t  &  -t
\end{array}
\right)
\]
for $x\in T_0\setminus\mathcal C$, and 
  \[
   D\Lambda_t(x)=
\left(
\begin{array}{cc}
  - t  & t  \\
  - t  &  -t
\end{array}
\right)
\]
for $x\in T_1\setminus\mathcal C$. From here we deduce that for all $x\in T\setminus\mathcal C$ we have
 \begin{equation}\label{caza}
\|D\Lambda_t^{-1}(x)\|\le \frac{1}{2t}.
\end{equation}
Now take $R=T$ and $\{R_i^t\}_{i=1}^{8}$ the (Lebesgue mod 0) partition of $R$ given by the domains of smoothness of $\Lambda_t^3$. 
\begin{figure}[!ht]
\begin{minipage}{1 \linewidth}
\centering
\subfigure[]{
\centering
\begin{tikzpicture}[xscale=0.28,yscale=0.28]
\draw [darkgray,thick] (0.014142151,-4.4041424) -- (8.814142,4.395858) -- (17.614141,-4.4041424)-- (0.014142151,-4.4041424);
\draw [darkgray,thick] (8.814142,4.395858) -- (8.814142,-4.4041424);
\end{tikzpicture}
}
\hspace{-3mm}
\subfigure[]{
\centering
\begin{tikzpicture}[xscale=0.28,yscale=0.28]
\draw [darkgray,thick] (0.014142151,-4.4041424) -- (8.814142,4.395858) -- (17.614141,-4.4041424)-- (0.014142151,-4.4041424);
\draw [darkgray,thick] (8.814142,4.395858) -- (8.814142,-4.4041424);
\draw [darkgray,thick] (5.034142,0.59585786) -- (8.794142,-3.0841422) -- (12.694142,0.5158579);
\end{tikzpicture}
}
\hspace{-3mm}
\subfigure[]{
\tiny{
\centering
\begin{tikzpicture}[xscale=0.28,yscale=0.28]
\draw [darkgray,thick] (0.014142151,-4.4041424) -- (8.814142,4.395858) -- (17.614141,-4.4041424)-- (0.014142151,-4.4041424);
\draw [darkgray,thick] (8.814142,4.395858) -- (8.814142,-4.4041424);
\draw [darkgray,thick] (5.034142,0.59585786) -- (8.794142,-3.0841422) -- (12.694142,0.5158579);
\draw [darkgray,thick] (6.014142,1.5958579) -- (11.614142,1.5958579);
\draw [darkgray,thick] (4.014142,-0.40414214) -- (4.014142,-4.4041424);
\draw [darkgray,thick] (13.614142,-0.40414214) -- (13.614142,-4.4041424);

\end{tikzpicture}
}
}
\caption{\em Smoothness domains: (a)  for $ \Lambda_t $\quad (b)   for $ \Lambda_t^2 $\quad (c)   for $ \Lambda_t^3 $}
\label{amigao}
\end{minipage}
\end{figure}
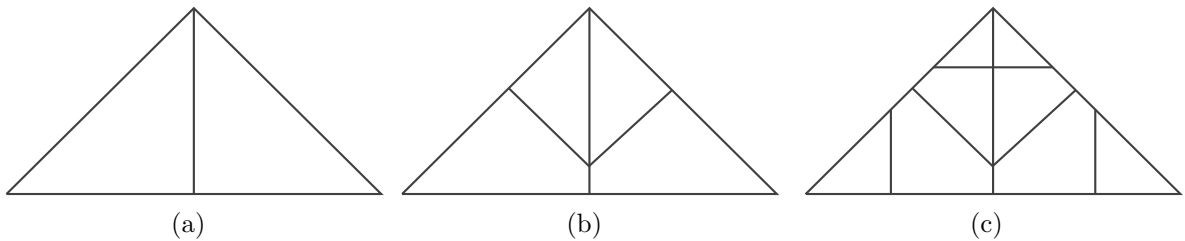
From \eqref{caza} we easily deduce that 
  \begin{equation*}\label{caza2}
\|(D\Lambda_t^3)^{-1}(x)\|\le \frac{1}{8t^3}:=\sigma_t<1
\end{equation*}
 and so property (P$_1$) holds for each $t\in[\tau,1]$; recall that $\tau\approx 0.88$. Since $\Lambda_t^3$ is linear on each domain of smoothness, then it has 0 distortion. Thus we  obtain property (P$_2$) with $D_t=0$, for each $t\in[\tau,1]$.

Let us now check (P$_3$). As  each $\Lambda_t^3$ is linear on each $R_i^t$ and preserves angles, it is enough to obtain the geometric property (P$_3$) for the domains $R_i^t$'s instead of their images.
 Since the pre-image of the critical set $\mathcal C$ delimits the boundary of the domains of smoothness, it easily follows that the boundary of each $R_i^t$ is formed by at most five straight line segments with slope $-1$, 0, 1 or $\infty$ meeting at an angle at least~$\pi/4$; see Figure~\ref{amigao}(c). 
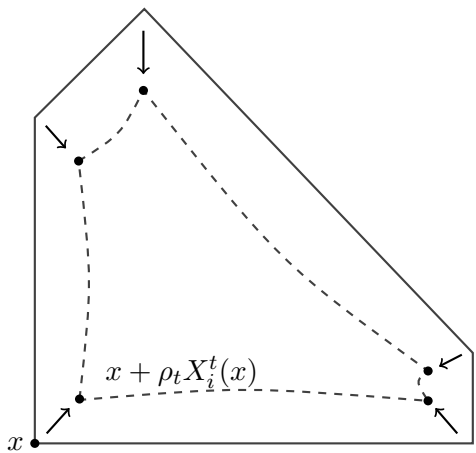
\begin{figure}[!ht]
\centering
\begin{minipage}{0.9 \linewidth}
\centering
{\small
\begin{tikzpicture}[xscale=0.6,yscale=0.6]
\draw [darkgray,thick] (0.02,-4.80434) -- (0.02,2.39566) -- (2.42,4.79566) -- (9.62,-2.8043401) -- (9.62,-4.80434) -- (0.02,-4.80434);
\draw [darkgray,thick,dashed] (1.0,-3.82434) .. controls (1.28,-1.2443401) .. (0.98,1.4356599) .. controls (1.9,2.01566) .. (2.4,2.9956598) .. controls (5.4,-0.82434005) .. (8.64,-3.20434) .. controls (8.36,-3.42434) .. (8.64,-3.86434) .. controls (4.6,-3.5443401) .. (1.0,-3.84434);

\fill[black] (0.02,-4.80434) circle (1mm); 
\fill[black] (1.0,-3.82434) circle (1mm); 
\fill[black] (0.98,1.4356599) circle (1mm); 
\fill[black] (2.4,2.9956598) circle (1mm); 
\fill[black] (8.64,-3.20434) circle (1mm); 
\fill[black] (8.64,-3.86434) circle (1mm); 
\node[black] at (2.6,-3.22434) {$\quad\quad x + \rho_t X_i^t(x) $};
\node[black] at (-0.4,-4.80434) {$ x $};

\draw[black,thick,->] (0.28,-4.58434) -- (0.78,-4.02434);
\draw[black,thick,->] (0.26,2.2156599) -- (0.7,1.71566);
\draw[black,thick,->] (2.4,4.31566) -- (2.4,3.35566);
\draw[black,thick,->] (9.38,-2.84434) -- (8.88,-3.10434);
\draw[black,thick,->] (9.28,-4.58434) -- (8.8,-4.02434);

\end{tikzpicture}
}
\caption{\em A long  branch for $ \Lambda_t^3 $}
\label{dois}
\end{minipage}
\end{figure}
Then, it is not hard to check that for every $t\in[\tau,1]$ and $i=1,\dots, 8$ there is a piecewise $C^1$ unitary vector filed $X_i^t$ in $\partial R_i^t$ such that
 $$|\sin\angle(v,X_i^t(x))|\geq
\sin\frac\pi8:=\beta_t$$
for every $x\in \partial R_i^t$ and  $v\in T_x\partial R_i^t\setminus\{0\}$; see Figure~\ref{dois}.
To prove the existence of $\rho_t$ is is enough to observe that the domains of smoothness of $\Lambda_t^3$ depend continuously on the parameter $t$ as illustrated in Figure~\ref{tres}, and so it is possible to choose an uniform value of $\rho$ such that~(P$_3$) holds for each $t\in[\tau,1]$.


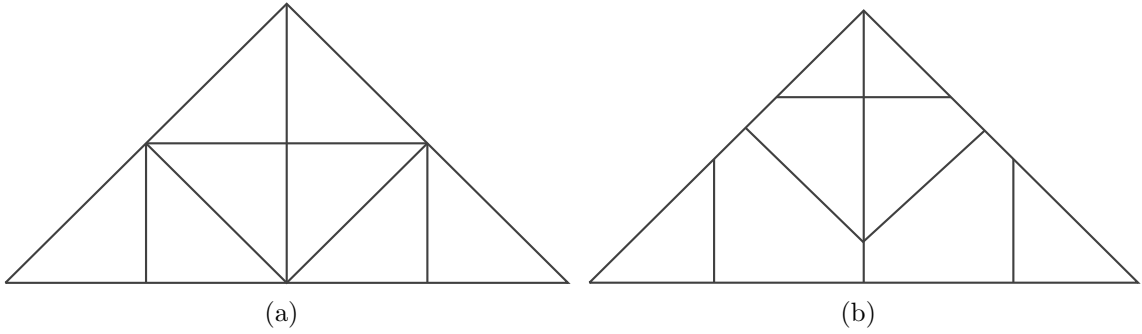
\begin{figure}[!ht]
\begin{minipage}{1 \linewidth}
\centering
\subfigure[]{
\centering
\begin{tikzpicture}[xscale=3.7,yscale=3.7]
\draw [darkgray,thick] (0,0) -- (2,0) -- (1,1) -- (0,0);
\draw [darkgray,thick] (1,0) -- (1,1) ;
\draw [darkgray,thick] (0.5,0.5) -- (1,0) -- (1.5,0.5) ;
\draw [darkgray,thick] (0.5,0.5) -- (1.5,0.5) ;
\draw [darkgray,thick] (0.5,0.5) -- (0.5,0) ;
\draw [darkgray,thick] (1.5,0.5) -- (1.5,0) ;
\end{tikzpicture}
}
\hspace{-3mm}
\subfigure[]{
\centering
\begin{tikzpicture}[xscale=0.41,yscale=0.41]
\draw [darkgray,thick] (0.014142151,-4.4041424) -- (8.814142,4.395858) -- (17.614141,-4.4041424)-- (0.014142151,-4.4041424);
\draw [darkgray,thick] (8.814142,4.395858) -- (8.814142,-4.4041424);
\draw [darkgray,thick] (5.034142,0.59585786) -- (8.794142,-3.0841422) -- (12.694142,0.5158579);
\draw [darkgray,thick] (6.014142,1.5958579) -- (11.614142,1.5958579);
\draw [darkgray,thick] (4.014142,-0.40414214) -- (4.014142,-4.4041424);
\draw [darkgray,thick] (13.614142,-0.40414214) -- (13.614142,-4.4041424);
\end{tikzpicture}
}
\caption{\em Domains of smoothness: (a) for $ \Lambda_1^3 $\quad (b) for  $ \Lambda_\tau^3$}
\label{tres}
\end{minipage}
\end{figure}

Altogether this shows that there are $0<\lambda<1$ and $K>0$ such that   
  $$\sigma_t\left(1+\frac1{\beta_t}\right)\le \frac1{8t^3}\left(1+\frac1{\sin(\pi/8)}\right)\le\lambda$$ 
for every $t\in[\tau,1]$ (recall that $\tau\approx 0.88$)  and 
    $$ D_t+\frac{1}{\beta_t\rho_t}+\frac{D_t}{\beta_t} =\frac1{\rho\sin(\pi/8)} := K,$$
    thus having proved \eqref{eq.dt} and hence Theorem~\ref{main}.


%

\begin{thebibliography}{00}

\bibitem{A96}
K.~Adl-Zarabi.
\newblock Absolutely continuous invariant measures for piecewise expanding
  {$C^2$} transformations in {${\bf R}^n$} on domains with cusps on the
  boundaries.
\newblock {\em Ergodic Theory Dynam. Systems}, 16(1):1--18, 1996.

\bibitem{A00}
J.~F. Alves.
\newblock {SRB} measures for non-hyperbolic systems with multidimensional
  expansion.
\newblock {\em Ann. Sci. \'Ecole Norm. Sup. (4)}, 33(1):1--32, 2000.

\bibitem{A04}
J.~F. Alves.
\newblock Strong statistical stability of non-uniformly expanding maps.
\newblock {\em Nonlinearity}, 17(4):1193--1215, 2004.

\bibitem{AV02}
J.~F. Alves and M.~Viana.
\newblock Statistical stability for robust classes of maps with non-uniform
  expansion.
\newblock {\em Ergodic Theory Dynam. Systems}, 22(1):1--32, 2002.

\bibitem{B99a}
J.~Buzzi.
\newblock Absolutely continuous invariant measures for generic
  multi-dimensional piecewise affine expanding maps.
\newblock {\em Internat. J. Bifur. Chaos Appl. Sci. Engrg.}, 9(9):1743--1750,
  1999.
\newblock Discrete dynamical systems.

\bibitem{B00a}
J.~Buzzi.
\newblock Absolutely continuous invariant probability measures for arbitrary
  expanding piecewise {$\bold R$}-analytic mappings of the plane.
\newblock {\em Ergodic Theory Dynam. Systems}, 20(3):697--708, 2000.

\bibitem{EG92}
L.~C. Evans and R.~F. Gariepy.
\newblock {\em Measure theory and fine properties of functions}.
\newblock Studies in Advanced Mathematics. CRC Press, Boca Raton, FL, 1992.

\bibitem{G84}
E.~Giusti.
\newblock {\em Minimal surfaces and functions of bounded variation}, volume~80
  of {\em Monographs in Mathematics}.
\newblock Birkh\"auser Verlag, Basel, 1984.

\bibitem{GB89}
P.~G{{\'o}}ra and A.~Boyarsky.
\newblock Absolutely continuous invariant measures for piecewise expanding
  {$C^2$} transformation in {${\bf R}^N$}.
\newblock {\em Israel J. Math.}, 67(3):272--286, 1989.

\bibitem{GB92}
P.~G{\'o}ra and A.~Boyarsky.
\newblock On functions of bounded variation in higher dimensions.
\newblock {\em Amer. Math. Monthly}, 99(2):159--160, 1992.

\bibitem{K79}
G.~Keller.
\newblock Ergodicit\'e et mesures invariantes pour les transformations
  dilatantes par morceaux d'une r\'egion born\'ee du plan.
\newblock {\em C. R. Acad. Sci. Paris S\'er. A-B}, 289(12):A625--A627, 1979.

\bibitem{LY73}
A.~Lasota and J.~A. Yorke.
\newblock On the existence of invariant measures for piecewise monotonic
  transformations.
\newblock {\em Trans. Amer. Math. Soc.}, 186:481--488, 1973.

\bibitem{PRT14}
A.~Pumari{\~n}o, J.~A. Rodr{\'\i}guez, J.~C. Tatjer, and E.~Vigil.
\newblock Expanding {B}aker maps as models for the dynamics emerging from
  3{D}-homoclinic bifurcations.
\newblock {\em Discrete Contin. Dyn. Syst. Ser. B}, 19(2):523--541, 2014.

\bibitem{PRT14a}
A.~Pumari{\~n}o, J.~A. Rodr{\'\i}guez, J.~C. Tatjer, and E.~Vigil.
\newblock Chaotic dynamics for 2-d tent maps, 2014. 
\newblock http://www.ma.utexas.edu/mp\_arc/c/14/14-8.pdf


\bibitem{PT06}
A.~Pumari{\~n}o and J.~C. Tatjer.
\newblock Dynamics near homoclinic bifurcations of three-dimensional
  dissipative diffeomorphisms.
\newblock {\em Nonlinearity}, 19(12):2833--2852, 2006.

\bibitem{PT07}
A.~Pumari{\~n}o and J.~C. Tatjer.
\newblock Attractors for return maps near homoclinic tangencies of
  three-dimensional dissipative diffeomorphisms.
\newblock {\em Discrete Contin. Dyn. Syst. Ser. B}, 8(4):971--1005, 2007.

\bibitem{S00}
B.~Saussol.
\newblock Absolutely continuous invariant measures for multidimensional
  expanding maps.
\newblock {\em Israel J. Math.}, 116:223--248, 2000.


\bibitem{T01}
J.~C. Tatjer.
\newblock Three-dimensional dissipative diffeomorphisms with homoclinic
  tangencies.
\newblock {\em Ergodic Theory Dynam. Systems}, 21(1):249--302, 2001.

\bibitem{T00}
M.~Tsujii.
\newblock Absolutely continuous invariant measures for piecewise real-analytic
  expanding maps on the plane.
\newblock {\em Comm. Math. Phys.}, 208(3):605--622, 2000.


\bibitem{T01b}
M.~Tsujii.
\newblock Absolutely continuous invariant measures for expanding piecewise
  linear maps.
\newblock {\em Invent. Math.}, 143(2):349--373, 2001.

\end{thebibliography}


\end{document}